\documentclass[12pt]{amsart}
\usepackage[active]{srcltx}
\usepackage{a4wide}
\usepackage{amsthm,amsfonts,amsmath,mathrsfs,amssymb}
\usepackage{dsfont}
\usepackage[T1]{fontenc}
\usepackage[utf8]{inputenc}
\usepackage{fourier}
\usepackage{pdfpages}
\usepackage{graphicx}

\def\a{\alpha}               
\def\d{\delta}       \def\la{\lambda}     \def\om{\omega}
       \def\t{\theta}       
                  
               \def\t{\theta}

       \def\vf{\varphi}

\def\G{\Gamma}

\def\D{{\mathbb D}}     
\def\C{{\mathbb C}}     
     
\def\R{{\mathbb R}}

\def\({\left(}       \def\){\right)}

\newtheorem{prop}{\sc Proposition}
\newtheorem{lem}[prop]{\sc Lemma}
\newtheorem{thm}[prop]{\sc Theorem}
\newtheorem{cor}[prop]{\sc Corollary}

\begin{document}
\title{Extremal problems for vanishing functions in Bergman spaces}
\author[A. Llinares]{Adri\'an Llinares}
\address{Departamento de Matem\'aticas, Universidad Aut\'onoma de
Madrid, 28049 Madrid, Spain}
\email{adrian.llinares@uam.es}
\author[D. Vukoti\'c]{Dragan Vukoti\'c}
\address{Departamento de Matem\'aticas, Universidad Aut\'onoma de
Madrid, 28049 Madrid, Spain} \email{dragan.vukotic@uam.es}
\thanks{The authors are partially supported by PID2019-106870GB-I00 from MICINN, Spain. The first author is supported by MICINN Fellowship, reference number FPU17/00040.}
\date{25 July, 2021.}
\begin{abstract}
We prove two sharp estimates for the subspace of a standard weighted Bergman space that consists of functions vanishing at a given point (with prescribed multiplicity).
\end{abstract}
\maketitle
\section{Introduction}
 \label{sec-intro}
\subsection{Statements of main results}
 \label{subsec-statements}
Let $\D$ denote the unit disk in the complex plane and $A^p_\a$ the Bergman space consisting of all analytic functions in $\D$ that are $p$-integrable with respect to the standard weight $(\a+1)(1-|z|^2)^\a$ and the normalized Lebesgue area measure $dA$ on $\D$, with $\a>-1$, equipped with the usual norm (when $1\le p<\infty$) or translation-invariant metric (when $0<p<1$).
\par
In this note we solve two linear extremal problems in $A^p_\a$, with restriction to the functions that vanish at a given point $a\in\D$. In the case of a simple zero at $a$, with $\vf_a(z) = (a-z)/(1-\overline{a}z)$, the sharp inequalities obtained can be written as follows:
\begin{equation}
 \max \left\{ \left\|\frac{f}{\vf_a}\right\|_{p,\a} \,:\,\|f\|_{p,\a}=1, f(a)=0 \right\} = \( \frac{\G \(\frac{p}{2}+2+\a\)}{\G(2+\a) \G \(\frac{p}{2}+1\)} \)^{1/p} \,.
 \label{eq-est-1}
\end{equation}
\begin{equation}
 \max \{|f^\prime(a)|\,:\,\|f\|_{p,\a}=1, f(a)=0\} = \( \frac{\G \(\frac{p}{2}+2+\a\)}{\G(2+\a) \G \(\frac{p}{2}+1\)} \)^{1/p} (1-|a|^2)^{-(1+(2+\a)/p)}\,,
 \label{eq-est-2}
\end{equation}
where $\G$ denotes the standard Euler Gamma function. We will actually  state and prove them in greater generality, for zeros of higher multiplicity.
\par
To the best of our knowledge, the sharp estimate \eqref{eq-est-1} is new while \eqref{eq-est-2} is new in the case $0<p<1$ but its proof is different from, and somewhat simpler than, the one given before.
\subsection{Motivation and some history}
 \label{subsec-motiv}
Similar or related problems have been studied extensively. For detailed treatments of extremal problems on the closely related classical Hardy spaces, we refer the reader to \cite[Chapter~8]{D}, \cite[Chapter~IV]{G} or \cite{Kh}. To list only a few references, extremal problems in Bergman spaces have been studied in \cite{F1}, \cite{F2}, \cite{F3}, \cite{KS}, \cite{OS}, \cite{R}, \cite{V1}, \cite{V2} for linear functionals and in \cite{V3} for norm-attaining operators. The related history of contractive zero-divisors in Bergman spaces is summarized in the monographs \cite{DS} and \cite{HKZ}. A characterization of isometric zero-divisors for general function spaces and a formal proof of their non-existence in Bergman spaces is given in \cite[Section~5]{ADMV}.
\par
Regarding our specific problems, we recall that the bounded operator of multiplication by $z$ (\textit{shift\/}) on $A^p$ is also bounded from below. That is, there exists a positive constant $C$ such that $\,\int_\D |g(z)|^p dA \le C \int_\D |z g(z)|^p dA\,$ for all $g\in A^p$; \textit{cf.\/} \cite[pp.~60-61]{DS}. Since it is easy to see that a function $f$ vanishing at the origin belongs to $A^p$ if and only if $f(z)=z g(z)$ for some $g\in A^p$, the above inequality can be restated in an  equivalent way as follows: if $f\in A^p$ and $f(0)=0$ then the function $f(z)/z$ is also in $A^p$ and $\|f(z)/z\|_p\le c \|f\|_p$ for some fixed $c>0$ and all such $f$. This is easily generalized to a zero of order at least $N$ at the origin and dividing out the monomial $z^N$ and to a Blaschke factor $\vf_a^N$ corresponding to a zero of order at least $N$ at $a\in\D$.
\par
From the point of view of extremal problems, it is a question of interest to determine the best possible constant $C$ in the norm inequality $\|f/\vf_a\|_{p,\a} \le C \|f\|_{p,\a}$ for the usual  weighted Bergman space $A^p_\a$. However, we have not been able to find it in the literature, even in the simplest case of a single Blaschke factor $\vf_a$, not even in the unweighted case $\a=0$ and for $a=0$.
In order to answer this question, the first result of this note gives a basic sharp estimate for general weighted Bergman spaces with positive radial weights and for dividing out the zero of order $N$ at the origin. By using a conformal change of variable, we then extend the estimate its form \eqref{eq-est-1} for the spaces with standard weights. The second main result solves (in greater generality than before) the closely related extremal problem \eqref{eq-est-2} for the point evaluation of the derivative of a function in the Bergman space that vanishes with prescribed multiplicity at the point in question.

\section{Preliminaries}
 \label{sect-prelim}
\subsection{Bergman spaces with radial weights}
 \label{subsec-wBergman-rad}
\par
In this section we review some basic facts needed in the sequel. Throughout the paper $\D$ will denote the unit disk in the complex plane: $\D=\{z=x+iy=re^{i\t}\,\colon\,|z|=r<1\}$ and $dA(z)=\pi^{-1}\,dx\,dy=\pi^{-1}\,r\,dr\,d\t$ the normalized Lebesgue area measure on $\D$. Let
$$
 M_p(r;f) = \( \frac{1}{2\pi} \int_0^{2\pi} |f(re^{i\t})|^p \,d\t \)^{1/p}
$$
denote the usual $p$-integral means over the circle of radius $r$ centered at the origin.
\par
We will say that $\om$ is a radial weight in the unit disk $\D$ if  $\om(z)=\om(|z|)$ for all $z\in\D$. Assuming that, in addition, such a weight is strictly positive and integrable, the \textit{weighted Bergman space\/} $A^p_\om$ is defined as the set of all funcions $f$ analytic in $\D$ such that
\begin{equation}
 \|f\|_{p,\,\om} = \( \int_\D |f (z)|^p \om(z) \,dA(z) \)^{1/p} = \( \int_0^1 2r \om (r) M_p^p(r;f)\,dr \)^{1/p} < \infty\,.
 \label{eq-polar}
\end{equation}
This expression defines a norm on $A^p_\om$ when $1\le p<\infty$. When $0<p<1$, $\|f\|_{p,\om}$ is still used although the space is no longer a normed space but $d_p(f,g)=\|f-g\|_{p,\om}^p$ defines a translation-invariant metric. The subharmonicity of $|f|^p$ and the mean-value property, together with \eqref{eq-polar} and a normal families argument,  easily show that point evaluations are bounded on $A^p_\om$ and this in turn implies that the space is actually complete.
\par
\subsection{Standard weighted Bergman spaces}
 \label{subsec-chng-sharp-est}
\par
Bergman spaces, including the standard weighted spaces, have been studied extensively; see monographs \cite{DS} and \cite{HKZ}. For  $\a>-1$, the \textit{standard radial weights\/} are given by $\om (z)=(\a+1) (1-|z|^2)^\a$. The corresponding weighted Bergman spaces are denoted by $A^p_\a$.
\par
The following estimate is well known; see \cite[Section~2.4]{DS} or \cite{V1}:
\begin{equation}
 |f(z)| \le (1-|z|^2)^{-(2+\a)/p} \|f\|_{p,\,\a}\,, \qquad z\in\D\,, \ f\in A^p_\a\,.
 \label{eq-Ap-est}
\end{equation}
Equality holds at $z=a$ only for the constant multiples of the functions
$$
 f_a(z) = \( \frac{1-|a|^2}{(1-\overline{a}z)^2} \)^{(2+\a)/p}\,.
$$
Both the estimate and some ingredients from its proof will be needed throughout the paper. Some details are in order. For $a\in\D$, we will use the customary notation $\vf_a$ for the following disk automorphisms $$
 \vf_a(z)=\frac{a-z}{1-\overline{a}z}\,, \quad z\in\D\,.
$$
These maps are involutions: $\vf_a(\vf_a(z))=z$ for all $z\in\D$. It is readily checked that they induce surjective linear isometries of the space $A^p_\a$ given by
\begin{equation}
 I_a f(z)= (\vf_a^\prime (z))^{(2+\a)/p} f(\vf_a(z))\,, \quad z\in\D\,.
 \label{eq-Ap-isom}
\end{equation}
These isometries show that it is enough to prove \eqref{eq-Ap-est} at the origin. When $1\le p<\infty$, this is a direct consequence of H\"older's inequality while when $0<p<1$ the use of the arithmetic-geometric mean inequality is required.

\section{Main results and their proofs}
 \label{sect-res-pfs}
\subsection{Chebyshev's inequality}
 \label{subsec-chebyshev}
\par
In what follows we will need the  classical Chebyshev inequality for monotone functions. We include a proof for the sake of completeness.
\par
\begin{lem} \label{lem-cheb}
Let $\mu$ be a positive finite measure on an interval $I\subset\R$ and $f$, $g$ non-decreasing functions on $I$ such that $f$, $g$, and $fg$ are integrable on $I$ with respect to $\mu$. Then
$$
 \int_I f g\,d\mu \ge \frac{1}{\mu(I)} \int_I f\,d\mu \int_I g\,d\mu\,. $$
Equality holds if and only if either $f$ or $g$ is constant $\mu$-almost everywhere.
\end{lem}
\begin{proof}
Since $(f(x)-f(y)) (g(x)-g(y))\ge 0$ for all $(x,y)\in I\times I$,  Fubini's theorem yields
$$
 0 \le \int_{I\times I} (f(x)-f(y)) (g(x)-g(y))\,d\mu(x)\,d\mu(y) =
 2 \mu (I) \int_{I} f(x) g(x)\,d\mu(x) - 2 \int_I f\,d\mu \int_I g\,d\mu\,.
$$
The case of equality is clear.
\end{proof}
\subsection{A sharp estimate for dividing out a single Blaschke factor}
 \label{subsec-sharp-est-div-zero}
The following result will be the key to our first main theorem. Note that, just like the remaining statements in this note, it is true regardless of the value of $p$; that is, it holds both when $A^p_\om$ is a Banach space and when it is a complete metric space. In other words, $\|f\|_{p,\,\om}$ is only used as a notation but no homogeneity property of the norm is needed in the proof.
\par
\begin{prop} \label{prop-rad-norm-est}
Let $\om$ be a positive radial weight defined in the unit disk $\D$ and let $0<p<\infty$. The following sharp norm inequality
\begin{equation}
  \left\|\frac{f(z)}{z^N}\right\|_{p,\,\om} \le \( \frac{\int_0^1 r \om (r)\,dr}{\int_0^1 r^{Np+1} \om (r)\,dr}\)^{1/p} \|f\|_{p,\,\om}
 \label{eq-wnorm-est}
\end{equation}
holds for all $f\in A^p_\om$ with a zero at the origin of order at least $N$, with equality only for the constant multiples of $f(z)=z^N$.
\end{prop}
\begin{proof}
Let $f\in A^p_\om$ be an arbitrary functon with a zero at the origin of order at least $N$. Then $f(z)=z^N g(z)$, where $g$ is analytic in $\D$. In order to apply Lemma~\ref{lem-cheb} with $I=[0,1]$ and $d\mu=2r \om(r) \,dr$, we need to check first that $M_p^p(r;g)$ is integrable with respect to this measure. Choose $\d$ such that $0<\d<1$. Then clearly $\int_0^\d M_p^p(r;g) 2 r \om(r) dr$ is finite since $M_p^p(r;g)$ is continuous, while
$$
 \int_\d^1 M_p^p(r;g) 2 r \om(r) dr \le \d^{-Np} \int_\d^1 r^{Np} M_p^p(r;g) 2 r \om(r) dr\,.
$$
Now, by Lemma~\ref{lem-cheb} we have
\begin{eqnarray*}
 \|f\|_{p,\,\om}^p &=& \int_0^1 r^{Np} M_p^p(r;g) 2r \om(r) \,d r
 \ge \frac{\int_0^1 r^{Np} 2r \om(r) \,d r}{\int_0^12r \om(r) \,d r} \int_0^1 M_p^p(r;g) 2r \om(r) \,d r
\\
 &= & \frac{\int_0^1 r^{Np+1} \om (r)\,d r}{\int_0^1\om (r) r\,d r} \|g\|_{p,\,\om}^p\,,
\end{eqnarray*}
which proves that $g\in A^p_\om$ and also the desired inequality. Equality holds if and only if $M_p(r;g)$ is a constant function of $r$. By the basic theory of Hardy spaces (\textit{cf\/}. \cite{D} or \cite{G}), this happens if and only if $g$ is constant. Thus, the extremal functions $f$ have the form claimed.
\end{proof}
\par
From Proposition~\ref{prop-rad-norm-est} we can deduce the following result for functions that vanish with order at least $N$ at $a$. Note that the optimal bound obtained does not depend on the point $a$. This is a consequence of conformal invariance.
\par
\begin{thm} \label{thm-aut-div}
Let $0<p<\infty$, $\a>-1$, and let $\om_\a (z)=(\a+1) (1-|z|^2)^\a$ be the standard radial weight. Denote by $A^p_\a$ the corresponding weighted Bergman space and by $\|f\|_{p,\a}$ the norm in it, as is usual. If $f\in A^p_\a$, $a\in\D$, and $f(a)=f^\prime(a)=\ldots=f^{(N-1)}(a)=0$ (for a positive integer $N$),  then
$$
 \left\|\frac{f}{\vf_a^N}\right\|_{p,\,\a} \le \( \frac{\G \(\frac{Np}{2}+2+\a\)}{\G(2+\a) \G \(\frac{Np}{2}+1\)} \)^{1/p} \|f\|_{p,\,\a}
$$
and the inequality is sharp. In particular, for the unweighted Bergman space $A^p$ and $N=1$ (a simple zero at $a$) this becomes:
$$
 \left\|\frac{f}{\vf_a}\right\|_p \le \(\frac{p+2}{2}\)^{1/p} \|f\|_p\,.
$$
\end{thm}
\begin{proof}
In the case $a=0$, the desired inequality reduces to
$$
  \left\|\frac{f(z)}{z^N}\right\|_{p,\,\a} \le \( \frac{\G \(\frac{Np}{2}+2+\a\)}{\G(2+\a) \G \(\frac{Np}{2}+1\)} \)^{1/p} \|f\|_{p,\,\a}\,,
$$
with equality only for the constant multiples of $f(z)=z^N$. This follows directly from Proposition~\ref{prop-rad-norm-est} by a simple computation involving polar coordinates:
$$
 \int_0^1 2r \om_\a(r) \,d r = 1\,, \qquad \int_0^1 2 r^{Np+1} \om_\a (r)\,dr = (\a+1) B\(\frac{Np}{2}+1,\a+1\)\,,
$$
together with the well-known identities for the Beta and Gamma functions:
$$
 B(p,q) = \frac{\G(p) \G(q)}{\G(p+q)}\,, \qquad \G(p+1)=p \G(p)\,.
$$
\par
In the general case, we employ the isometries defined by \eqref{eq-Ap-isom}. Since $f\in A^p_\a$ and has a zero of multiplicity at least $N$ at $a$, the function $I_a f$ is also a function in $A^p_\a$ with a zero of multiplicity at least $N$ at the origin. By inspection, we see that
$$
 I_a\(\frac{f}{\vf_a^N}\)(z) = \frac{I_a f(z)}{z^N}\,,
$$
hence $f/\vf_a^N$ is also an $A^p_\a$ function of the same norm. Using this and the inequality already proved for $a=0$, we obtain:
\begin{eqnarray*}
 \left\|\frac{f}{\vf_a^N}\right\|_{p,\,\a}  &=& \left\|I_a\(\frac{f}{\vf_a^N}\)\right\|_{p,\,\a}
 = \left\| \frac{I_a f(z)}{z^N} \right\|_{p,\,\a}
 \le \( \frac{\G \(\frac{Np}{2}+2+\a\)}{\G(2+\a) \G \(\frac{Np}{2}+1\)} \)^{1/p} \|I_a f\|_{p,\,\a}
\\
 &=& \( \frac{\G \(\frac{Np}{2}+2+\a\)}{\G(2+\a) \G \(\frac{Np}{2}+1\)} \)^{1/p} \|f\|_{p,\,\a}\,.
\end{eqnarray*}
As we already know, equality will hold only for the constant multiples of the function $f$ determined by the condition $I_a f(z) = z^N$, that is, for
\begin{equation}
 f(z) = I_a (z^N) = \vf_a (z)^N \( \vf_a^\prime (z) \)^{\frac{2+\a}{p}} = \frac{(a-z)^N (1-|a|^2)^\frac{2+\a}{p}} {(1-\overline{a}z)^{N+\frac{2(2+\a)}{p}}}
 \label{eq-extr-fcn}
\end{equation}
and its constant multiples.
\end{proof}
\par
Of course, Theorem~\ref{thm-aut-div} can be stated equivalently as follows.
\par
\textit{If $g\in A^p_\a$, $N\ge 1$, and $a\in\D$, then\/}
$$
 \left\| g \right\|_{p,\,\a} \le \( \frac{\G \(\frac{Np}{2}+2+\a\)}{\G(2+\a) \G \(\frac{Np}{2}+1\)} \)^{1/p} \|\vf_a^N g\|_{p,\,\a}
$$
\textit{and the inequality is sharp.}
\par
It should be noted that, even in the case $N=1$, the problem of norm computation when dividing out a Blaschke product of degree two (say, $\vf_a\vf_b$) is considerably more complicated. It is not difficult to see that a  repeated application of Theorem~\ref{thm-aut-div}, dividing out first one single factor and then the other, does not give a sharp bound. The norm computation in this case constitutes an interesting open problem in our opinion.

\subsection{A sharp estimate for the derivative}
 \label{subsec-est-deriv}
\par
Osipenko and Stessin \cite[Theorem~5]{OS} found different sharp estimates for linear combinations $|\la_0 f(a) + \la_1 f^\prime (a)|$ for arbitrary $f\in A^p$, depending on the values of the coefficients $\la_0$ and $\la_1$. The estimates change significantly after reducing the space to the subspace of functions vanishing at $a$, as was shown in \cite[Theorem~11]{V2} for $p\ge 1$. The result below extends this to the case $0<p<1$, as a consequence of Theorem~\ref{thm-aut-div}. It also allows us to give a somewhat simpler proof in the known case $1\le p<\infty$.
\par
\begin{cor} \label{cor-ptwse-est}
Let $0<p<\infty$ and $\a>-1$. If $f\in A^p_\a$, $a\in\D$, and $f(a)=f^\prime(a)=\ldots=f^{(N-1)}(a)=0$ for a positive integer $N$, then
$$
 \left|f^{(N)}(a)\right| \le N! \( \frac{\G \(\frac{Np}{2}+2+\a\)}{\G(2+\a) \G \(\frac{Np}{2}+1\)} \)^{1/p} \frac{ \|f\|_{p,\,\a}}{(1-|a|^2)^{N+(2+\a)/p}}\,.
$$
The inequality is sharp.
\end{cor}
\begin{proof}
Writing first $f(z)=(z-a)^N g(z)$, with $g$ analytic in $\D$, it is elementary that
$$
 \frac{f^{(N)}(a)}{N!}=g(a)\,.
$$
Next, in view of the sharp estimate \eqref{eq-Ap-est} and Theorem~\ref{thm-aut-div}, for every $z\neq a$ we have
$$
 \left|g(z) (1-\overline{a}z)^N\right| = \left|\frac{f(z)}{\vf_a(z)^N}\right| \le \frac{\left\| \frac{f}{\vf_a} \right\|_{p,\,\a}}{(1-|z|^2)^{(2+\a)/p}} \le
 \( \frac{\G \(\frac{Np}{2}+2+\a\)}{\G(2+\a) \G \(\frac{Np}{2}+1\)} \)^{1/p} \frac{\|f\|_{p,\,\a}}{(1-|z|^2)^{(2+\a)/p}}\,.
$$
Letting $z\to a$, we obtain the desired estimate.
\par
A direct computation shows that equality holds for the function given by \eqref{eq-extr-fcn}.
\end{proof}


\end{document}